\documentclass[12pt]{amsart}

\usepackage{amssymb}

\usepackage{enumitem}

\usepackage{graphicx}

\makeatletter
\@namedef{subjclassname@2010}{%
  \textup{2010} Mathematics Subject Classification}
\makeatother

\usepackage[T1]{fontenc}

\newtheorem{theorem}{Theorem}[section]

\newtheorem{lemma}[theorem]{Lemma}

\theoremstyle{definition}

\numberwithin{equation}{section}

\def\cA{{\mathcal A}}
\def\cB{{\mathcal B}}
\def\cC{{\mathcal C}}
\def\cD{{\mathcal D}}

\def\cI{{\mathcal I}}

\def\cN{{\mathcal N}}

\def\cS{{\mathcal S}}
\def\cT{{\mathcal T}}

\def\cW{{\mathcal W}}

\def\mand{\qquad \mbox{and} \qquad}

\usepackage{bbm}

 \newcommand{\RN}[1]{%
  \textup{\uppercase\expandafter{\romannumeral#1}}%
}

\begin{document}


\baselineskip=17pt


\title[Distribution of $\alpha n + \beta$ modulo 1]{Distribution of $\alpha n +\beta$ modulo 1 over integers free from large and small primes}

\author[K. H. Yau]{Kam Hung Yau}

\address{Department of Pure Mathematics, University of New South Wales,
Sydney, NSW 2052, Australia}
\email{kamhung.yau@unsw.edu.au}

\date{}

\begin{abstract}
For any $\varepsilon >0$, we obtain an asymptotic formula for the number of solutions $n \le x$ to
$$
\lVert \alpha n + \beta \rVert < x^{-1/4+\varepsilon} 
$$
where $n$ is $[y,z]$-smooth for infinitely many real numbers $x$. In addition, we also establish an asymptotic formula with an additional square-free condition on $n$. Moreover, if $\alpha$ is  quadratic irrational then the asymptotic formulas hold for all sufficiently large $x$.

 Our tools come from the Harman sieve which we adapt suitably to sieve for $[y,z]$-smooth numbers. The arithmetic information comes from estimates for exponential sums.  
\end{abstract}



\maketitle

\section{introduction}

A classical Theorem of Dirichlet~\cite[Theorem 185]{HW} states that if $\alpha$ is irrational then there exist infinitely many pairs of integers $(m,n)$ satisfying the inequality
\begin{equation} \label{dis1}
\left | \alpha - \frac{m}{n} \right | < \frac{1}{n^{2}}.
\end{equation} 
Such problems fall in the area of Diophantine approximation.
Equivalently by defining
 $\lVert x \rVert = \min_{v \in \mathbb{Z}} |x-v|$ to be the distance from $x$ to the nearest integer, (\ref{dis1}) implies
$$
\lVert \alpha n \rVert < n^{-1}
$$
for infinitely many positive integers $n$.
A natural extension to this problem is to consider if there exist infinitely many  solutions to
\begin{equation} \label{dio-ineq}
\lVert \alpha n + \beta \rVert < n^{-\kappa}.
\end{equation}
Here $\beta \in \mathbb{R}$, $\kappa >0$, $n \in \cN \subseteq \mathbb{N}$ and $\cN$ is some set of arithmetic interest.

Various results have been obtained when $\cN$ is the set of prime numbers. These include the results of Harman~\cite{H1,H2}, Heath-Brown \& Jia~\cite{Heath} and Jia~\cite{Jia}, which involve sieve methods and bounds of exponential sums, while Vaughan~\cite{V} obtained his result by applying what is now known as the {\it Vaughan Identity}~\cite{V2}, together with bounds for exponential sums.

We remark that Harman's sieve method first appeared when  Harman studied this exact problem~\cite{H1} to get the exponent $\kappa = 3/10$. For more details on the Harman sieve, see the monograph~\cite{H}. The best result to date is by Matom\"aki~\cite{M} with $\kappa < 1/3$ under the condition $\beta =0$ and employs Harman's sieve method where the arithmetic information comes from bounds for averages of Kloosterman sums.

Let $k>1$ be a fixed positive integer and $\cN$ the set of $k$th powers of primes. Baker \& Harman~\cite{BH} showed that we can take $\kappa < 3/20$ if $k=2$ and $\kappa < (3 \cdot 2^{k-1})^{-1}$ if $k \ge 3$. In particular, when $k=2$ this improves a result of Ghosh~\cite{Chosh}. Later, Wong~\cite{W} provided an improvement of Baker \& Harman~\cite{BH} in the range $3 \le k \le 12$.
 
When $\cN$ is the set of square-free numbers, the best result is due to Heath-Brown~\cite{Heath1} with $\kappa <2/3$ using an essentially elementary method. The result improves the previous work by Harman~\cite{H3} and Balog \& Perelli~\cite{BP}, who independently showed we can essentially take the exponent $\kappa < 1/2$.

In this paper, we consider the problem of establishing an asymptotic formula for the number of $[y,z]$-smooth $n \le x$ solutions  to~(\ref{dio-ineq}) (numbers with prime factors in the interval $[y,z]$) with $\kappa =1/4 -\varepsilon$, where $\varepsilon >0$. We also consider a hybrid problem that interpolates between square-free and $[y,z]$-smooth integers. 

We note that in the special case of smooth numbers we are able to obtain a non-trivial lower bound immediately. Indeed, for any fixed $\varepsilon >0$, consider the set
$$
\{n=ab \le x : p|ab \implies p < x^{\varepsilon}, x^{1/3} \le a < 2 x^{1/3} \}.
$$
By applying Lemma~\ref{expansion} and our Type II information (Lemma~\ref{TypeI,II-bound} with~(\ref{le 1/2})), we can show immediately that for infinitely many $x$ chosen correctly there are at least $x^{2/3+\varepsilon-o(1)}$ integers up to $x$ which are $x^{\varepsilon}$-smooth  and satisfies
$$
\lVert \alpha n + \beta \rVert < x^{-1/3 +\varepsilon}.
$$
The statement remains valid if we additionally require $n$ to be square-free.

\section{Notation}
For complex valued functions $f$ and $g$, we use the notation $f =O(g)$ and $f \ll g$ to mean there exists an absolute constant $C > 0$ such that $ |f(x)|\le C |g(x)|$ for all sufficiently large $x$.

Write $m \sim M$ when $ M \le m < 2M$, $e(x)= \exp(2 \pi i x)$ and let $\mu$ be the M\"obius function. 
Given a positive integer $n \ge 2$, we denote by $P^-(n)$ and $P^+(n)$ the smallest and largest prime factor of $n$ respectively. We always assume $p$ with or without sub/superscript are primes.

\section{Main Results}
For any real numbers $0<y \le z \le x$, we denote by
$$
S(x;y,z)= \{a \in [1,x] \cap \mathbb{N}: p|a \implies p \in [y,z] \}
$$
 the set of all integers in $[1,x]$ all of whose prime factors are in $[y,z]$. Moreover, we denote the cardinality of $S(x;y,z)$ by $\Psi(x;y,z)$.

For any positive integer $A > 0$, we write 
\begin{align*}
\cI(A) = \{ \alpha \in \mathbb{R} \backslash \mathbb{Q} : \alpha =[a_0; a_1 , \ldots ], |a_j| \le A, j \ge 0 \}.
\end{align*}
We note that $\cup_{A \in \mathbb{N}} \cI(A)$ contains the set of all quadratic irrationals. We can now state our first result.

\begin{theorem} \label{smooth}
There exists an increasing sequence  $( x_{k} )_{k \in \mathbb{N}}$ of positive integers such that if $2 \le y < x_k^{1/2 }$, $y < z \le x_k$, $\varepsilon > 0$ and $\delta = x_k^{-1/4 +\varepsilon}$, then 
\begin{equation} \label{y-z smooth eqn}
\sum_{\substack{n \in S(x_k;y,z) \\ \lVert \alpha n + \beta \rVert < \delta \\ }} 1   =  2\delta \Psi(x_k;y,z)  + O(x_k^{ 3/4+ \varepsilon /2 +o(1) }).
\end{equation}
Moreover,~$\normalfont{(\ref{y-z smooth eqn})}$ \textit{holds for the sequence $( x_{k})_{k \in \mathbb{N}} = (k)_{k \in \mathbb{N}}$  uniformly for all $\alpha \in \cI(A)$ and any fixed positive integer $A$}.
\end{theorem}

For any real numbers $0 <y \le z \le x$, we denote by
$$
S^*(x;y,z) =\{ a \in [1,x] \cap \mathbb{N} : \mu^2(a)=1, p|a \implies p \in [y,z] \}
$$
the set of all square-free integers in $[1,x]$ all of whose prime factors are in $[y,z]$. We also denote the cardinality of $ S^*(x;y,z)$ by $\Psi^*(x;y,z)$. The next theorem is essentially Theorem~\ref{smooth}, where we also require the integers we are counting to be square-free.

\begin{theorem} \label{y-z smooth}
There exists an increasing sequence $( x_{k} )_{k \in \mathbb{N}}$  of positive integers such that if $2 \le y < x_k^{1/2 }$, $y < z \le x_k$, $\varepsilon > 0$ and $\delta = x_k^{-1/4 +\varepsilon}$, then we have
\begin{equation} \label{s.f y-eqn}
\sum_{\substack{n \in S^*(x_k;y,z) \\ \lVert \alpha n + \beta \rVert < \delta \\ }} 1   =  2\delta \Psi^*(x_k;y,z)  + O(x_k^{ 3/4+ \varepsilon/2 +o(1) }).
\end{equation}
Moreover,~$\normalfont{(\ref{s.f y-eqn})}$ \textit{holds for the sequence $( x_{k})_{k \in \mathbb{N}} = (k)_{k \in \mathbb{N}}$  uniformly for all $\alpha \in \cI(A)$ and any fixed positive integer $A$}.
\end{theorem}

If we fix $0 < u_2 < u_1$ with $2 < u_1 $, $u_2 < \lfloor u_1 \rfloor$ and set $y = x_k^{1/u_1 }$ and $z=x_k^{1/u_2 }$ then both  $\Psi(x_k;y,z)$, $\Psi^*(x_k;y,z)$ in the main term of~(\ref{y-z smooth eqn}) and~(\ref{s.f y-eqn}) respectively are bounded below by the number of integers $n=p_1 \ldots p_j$ that are products of $j = \lfloor u_1 \rfloor$ distinct primes  with $p_i \in [y,z]$. This gives the lower bound 
$$
x_k^{1-o(1)} < \Psi(x_k;y,z),\Psi^*(x_k;y,z).
$$ 
It follows that our Theorem~\ref{smooth} and~\ref{y-z smooth} are non-trivial in this region.

We note that by~\cite[Theorem 1]{F} of Friedlander, we can obtain an asymptotic formula for $\Psi(x;y,z)$ in certain regions. We also mention that Saias~\cite{S1,S2,S3} has  extensively studied this quantity. In particular, our  $\Psi(x_k;y,z)$ is $\theta(x_k,z,y)$ or $\Theta(x_k,z,y)$ in the notation of Friedlander and Saias respectively. 

We remark that we assume $y < x_k^{1/2}$ in Theorem~\ref{smooth} and~\ref{y-z smooth} since if $y \ge x_k^{1/2}$ and $z \ge x_k$ then both $S(x_k;y,z)$ and $S^*(x_k;y,z)$ essentially count primes in the interval $[y,x_k]$ and the result of Harman~\cite[Theorem 3.2]{H} covers this case with $\delta = x_k^{-1/4+\varepsilon}$.

It is easy to see that the proof of Theorem~\ref{smooth} can be adjusted to prove Theorem~\ref{y-z smooth}, so we will only give full details for the proof of Theorem~\ref{smooth}.

\section{Preparations}

For any $\delta >0$, we define
$$
\chi(r) = 
\begin{cases}
1 &\mbox{ if } \lVert r \rVert < \delta, \\ 
0 &\mbox{ otherwise.}
\end{cases}
$$
We recall a Lemma from~\cite[Chapter 2]{B}, which provides a finite Fourier approximation to $\chi$. This converts the problem of detecting solutions to~(\ref{dio-ineq}) to a problem about estimates for exponential sums.
\begin{lemma} \label{expansion}
For any positive integer $L$, there exist complex sequences $( c_{\ell}^{-} )_{|\ell |\le L}$ and $ ( c_{\ell}^{+} )_{|\ell |\le L} $ such that
$$
2 \delta - \frac{1}{L+1} + \sum_{0 < |\ell | \le L} c_{\ell}^{-} e(\ell \theta) \le \chi ( \theta) \le 2 \delta + \frac{1}{L+1} + \sum_{0 < | \ell| \le L} c_{\ell}^{+} e(\ell \theta)
$$
where
\begin{equation*} \label{fourier-expansion-coeff}
| c_{\ell}^{\pm} |  \le \min \left \{ 2 \delta + \frac{1}{L+1}, \frac{3}{2 \ell} \right \}.
\end{equation*}
\end{lemma}

Let $\alpha ,x \in \mathbb{R}$ be such that $|\alpha - a/q| < 1/q^2$ where $1 \le a \le q \ll x^{O(1)}$ and $(a,q)=1$. Let  $(a_{m})_{m \in \mathbb{N}}$, $(b_{n})_{n \in \mathbb{N}}$, $(c_{\ell})_{\ell \in \mathbb{N} }$ be complex sequences satisfying 
$$
|a_{m}| \le \tau(m), \quad |b_{n}| \le \tau(n), \quad 
|c_{\ell}| \le \min \Big \{ 2 \delta + \frac{1}{\lfloor x \rfloor +1}, \frac{3}{2\ell} \Big \}
$$
whenever $m,n, \ell \in \mathbb{N}$. We will also use the classical divisor bound $\tau(r) \ll x^{o(1)}$ for $r \le x$. Lastly for $\varepsilon >0$, we denote $\delta = x^{-\frac{1}{4} + \varepsilon}$.

The next two Lemmas provide our Type I and II information respectively. They can be obtained by following the method of~\cite[Section 2.3]{H}.

\begin{lemma} \label{typeI exp-bound}
We have
\begin{align*}
 \sum_{\ell \le x}  c_{\ell}  & \sum_{\substack{mn \le x \\ m \sim  M }}  a_{m}  e(\alpha \ell mn) \ll  (M +  xq^{-1} + \delta q)x^{o(1)}  .
\end{align*}
\end{lemma}

\begin{lemma} \label{TypeI,II-bound}
We have
\begin{equation} \label{le 1/2}
\sum_{\ell \le x}  c_{\ell}  \sum_{\substack{mn \le x \\ m \sim M }} a_{m} b_{n} e(\alpha \ell mn) \ll 
x^{1+o(1)}\left( \frac{\delta}{M}  + \frac{M}{x}  + \frac{1}{q} + \frac{q \delta }{x} \right)^{ \frac{1}{2} } 
\end{equation}
for $M \ll x^{1/2 }$ and
\begin{equation} \label{ge 1/2}
\sum_{\ell \le x}  c_{\ell}  \sum_{\substack{mn \le x \\ m \sim M }} a_{m} b_{n} e(\alpha \ell mn) \ll x^{1+o(1)} \left(\frac{M \delta}{x} + \frac{1}{M} + \frac{1}{q} + \frac{q\delta }{x} \right)^{ \frac{1}{2} } 
\end{equation}
for $M \gg x^{1/2 }$.
\end{lemma}

\subsection{Estimates for Type I \& II sum} \label{sec:type I, II estimate}

Denote
$$
\cB = \{n \in \mathbb{N} : 2 \le n \le x \} \mand 
  \cA = \{ n \in \cB: \lVert \alpha n + \beta \rVert < \delta \}.
$$

We will first state our Type I estimate.

\begin{lemma}[Type I estimate] \label{type one estimate}
For $q = x^{2/3}$ we have
$$
\sum_{\substack{mn \in \cA \\ m \sim M}} a_m = 2 \delta \sum_{\substack{mn \in \cB \\ m \sim M}} a_m + O(x^{3/4 +\varepsilon/2 })
$$
whenever $M \ll x^{ 3/4 }$.
\end{lemma}

\begin{proof}
By Lemma~\ref{expansion} with $L = \lfloor  x \rfloor$, we get 
$$
\sum_{\substack{mn \in \cA \\ m \sim M}} a_m = \sum_{\substack{mn \in \cB \\ m \sim M}} a_m \chi(\alpha mn + \beta ) = 2\delta \sum_{\substack{mn \in \cB \\ m \sim M}} a_m +O(E_1 + E_2)
$$
where
\begin{align*}
E_1 & = \frac{1}{L+1} \sum_{\substack{mn \le x \\ m \sim M}} a_m  \ll x^{o(1)}, \\
E_2 & = \sum_{0 < |\ell| \le L} |c_{\ell}^{\pm}| \cdot \Bigg |\sum_{\substack{mn \le x \\ m \sim m}} a_m e(\ell(\alpha mn + \beta)) \Bigg |.
\end{align*}
Clearly there exists $\xi_{\ell} \in \mathbb{C}$ with $|\xi_{\ell}|=1$ such that
\begin{align*}
E_{2} 
& = \sum_{0 < |\ell| \le L} \xi_{\ell}c_{\ell}^{\pm} \sum_{\substack{mn \le x \\ m \sim M  }} a_{m} e( \alpha \ell  mn). 
\end{align*}
 Applying Lemma~\ref{typeI exp-bound} immediately gives the bound
$$
E_2 \ll (M + xq^{-1} + \delta q)x^{o(1)}  \ll x^{ \frac{3}{4} + \frac{\varepsilon}{2} }
$$
whenever $M \ll x^{3/4 }$.
\end{proof}

Next, we state our Type II estimate.

\begin{lemma}[Type II estimate] \label{type two estimate}
For $q = x^{ 2/3 }$ we have
$$
\sum_{\substack{mn \in \cA \\ x^{\gamma} \le m \le x^{\gamma+\tau} }} a_m b_n = 2\delta \sum_{\substack{mn \in \cB \\ x^{\gamma} \le m \le x^{\gamma+\tau}}} a_m b_n + O(x^{ 3/4 + \varepsilon/2 + o(1) })
$$
uniformly for $1/4 \le \gamma$, $\gamma +\tau \le 3/4$.
\end{lemma}

\begin{proof}
We follow the method of Lemma~\ref{type one estimate}. Partition the summation over $m$ into dyadic intervals and apply Lemma~\ref{TypeI,II-bound}.
\end{proof}

\subsection{Sieve estimates}

For positive real numbers $w, u, v$, we denote
$$
P(w) = \prod_{p < w}p \mand P(u,v] = \prod_{u < p\le v} p.
$$
We also set 
$$
Y=x^{ 3/4+ \varepsilon/2 + o(1)}.
$$

For any set $\cA \subseteq [2,x]$ of integers and any positive integer $s$, we denote
$$
\cA_s = \{ n : ns \in \cA  \}.
$$
For real numbers $0<y\le z \le x$, we denote
$$
\cS(\cA;y,z) = \# \{  a \in \cA : p|a \implies p \in [y,z]     \}
$$
and in the special case $z=x$, we denote
$$
\cS(\cA;y,z) = \cS( \cA; y).
$$
We state a variant of the Buchstab identity for $[y,z]$-smooth numbers which is based on taking out the largest prime factor of the integers we are counting.

\begin{lemma} \label{lem:Buchstab}
For any $2 \le  y \le z$ and any set $\cA \subseteq [2,x] $ of integers, we have
$$
\cS(\cA;y,z) = \sum_{y \le p \le z} \cS(\cA_p;y,p) + \cS(\cC; x^{1/2}) +O(x^{ 1/2 })
$$
where $\cC =\{a \in \cA: y \le a \le z \}$.
\end{lemma}

\begin{proof}
Take $a \in \{ a \in \cA : p|a \implies p \in [y,z] \}$. If $a$ is a prime then $a$ is counted in $\{a \in \cA: y \le a \le z , a \mbox{ is prime} \}$. Otherwise $a$ has at least two prime factors and we can write $a = P^+(a)n$ where $n>1$ and the prime factors of $n$ lie in the interval $[y,P^+(a)]$; the result follows immediately.
\end{proof}

Next we state three lemmas which give sieve estimates for different regions. In particular, the proofs will rely on ingredients coming from the Harman sieve~\cite[Lemma 2]{H2}.

Our first sieve estimate is essentially based on an application of our Type II estimate (Lemma~\ref{type two estimate}).

\begin{lemma} \label{1/4 3/4}
For $x^{ 1/4 } \le y < z \le x^{ 3/4 }$, we have
\begin{equation*} 
\sum_{y  \le p \le z} \cS (\cA_{p}; y,p)  = 2\delta \sum_{y  \le p \le z} \cS (\cB_{p}; y,p) +  O(Y).
\end{equation*}
\end{lemma}

\begin{proof}
We have
\begin{align*}
 \sum_{y  \le p \le z} \cS (\cA_{p}; y,p)  
& = \sum_{y \le  p \le z } \sum_{\substack{np \in \cA \\ (n,P(y)P(p,x])=1}} 1  = \sum_{ \substack{mn \in \cA \\ y \le m \le z }}     a(m,n).
\end{align*}
Here
\begin{align*}
a(m,n) & =   \mathbbm{1}_{\mathbb{P}}(m) \sum_{(n,P(y)P(m,x])=1} 1  =    \mathbbm{1}_{\mathbb{P}}(m) \sum_{\substack{d|n \\ d|P(y)P(m,x]}} \mu(d)
\end{align*}
where $\mathbbm{1}_{\mathbb{P}}(\cdot)$ is the characteristic function for the primes. We may assume that $d$ is square-free, as otherwise it does not contribute to the sum $a(m,n)$. 

Write
$$
a(m,n) =  \mathbbm{1}_{\mathbb{P}}(m) \sum_{\substack{d_1 d_2|n \\ d_2|P(y) \\ P^-(d_1) > m}} \mu(d_1 d_2).
$$
We may assume $d_1 >1$; the case $d_1=1$ follows immediately.
We now get rid of the condition $P^-(d_1) > m$ by appealing to the truncated Perron formula 
$$
\frac{1}{\pi} \int_{-T}^{T} e^{i \gamma t} \frac{\sin \rho t}{t} dt =
\begin{cases}
1 + O(T^{-1}(\rho - |\gamma|)^{-1} ) &  \mbox{ if $|\gamma | \le \rho$,}\\
O(T^{-1}(|\gamma| - \rho)^{-1} ) & \mbox{ if $|\gamma| > \rho$}.
\end{cases}
$$
Applying the above formula with $\rho = \log (P^-(d_1) -\frac{1}{2})$, $\gamma = \log m$ and $T = x^2 \delta^{-1}$, we have
\begin{align*}
 \sum_{ \substack{mn \in \cA \\ y \le m \le z }}    a(m,n) = M(\cA) + O(R).
\end{align*}
Here
\begin{align*}
M (\cA) & =  \frac{1}{\pi} \int_{-T}^{T} \sum_{\substack{mn \in \cA \\ y \le m \le z}} \mathbbm{1}_{\mathbb{P}}(m) e^{i\gamma t} \Bigg (\sum_{\substack{d_1 d_2|n \\ d_2|P(y) }} \mu(d_1 d_2) \sin  ( \rho t ) \Bigg ) \frac{dt}{t}\\
R & =T^{-1} \sum_{\substack{mn \in \cA \\ y \le m \le z}} \mathbbm{1}_{\mathbb{P}}(m) \sum_{\substack{d_1 d_2|n \\ d_2|P(y) }}  \frac{\mu(d_1d_2)}{|\log (P^-(d)- \frac{1}{2}) - \log m|}.
\end{align*}
We will consider the remainder term first. By the mean value theorem,
$$
\frac{1}{\log(P^-(d)-\frac{1}{2}) - \log m} = \frac{\eta}{P^-(d)-\frac{1}{2}-m} 
$$
where $\eta \in [m,P^-(d)-\frac{1}{2}]$ or $[P^-(d)-\frac{1}{2},m]$. In any case, we have $\eta \le \max \{m, n- \frac{1}{2} \}$, so we bound
$$
\frac{\eta}{P^-(d) -\frac{1}{2}-m} \ll \max \{ m,n  \} \le x.
$$ Therefore $R \ll T^{-1}x^{2+o(1)}\ll \delta x^{o(1)}$. 

It remains to estimate the main term. Note that the integral in the main term between $-1/T$ and $1/T$ can be trivially bounded by $\ll T^{-1} x^{1+o(1)} \ll \delta x^{-1}$. Applying our Type II estimate (Lemma~\ref{type two estimate}) in the integral over $\mathfrak{R}(T) = (-T,-1/T) \cup (1/T, T)$, we get
\begin{align*}
& M(\cA) \\
 & =  \frac{2\delta}{\pi}  \int_{\mathfrak{R}(T)}  \sum_{\substack{mn \in \cB \\ y \le m \le z}} \mathbbm{1}_{\mathbb{P}}(m) m^{it} \sum_{\substack{d_1 d_2|n \\ d_2|P(y) }} \mu(d_1 d_2) \sin \Big (t \log \Big (P^-(d_1)-\frac{1}{2} \Big ) \Big )  \frac{dt}{t}\\
&  \quad + O \Big (Y \int_{\mathfrak{R}(T)} \frac{dt}{t}  \Big )   + O(\delta x^{-1})\\
& = \frac{2\delta}{\pi} \int_{-T}^{T}  \sum_{\substack{mn \in \cB \\ y \le m \le z}} \mathbbm{1}_{\mathbb{P}}(m) m^{it} \sum_{\substack{d_1 d_2|n \\ d_2|P(y)}} \mu(d_1 d_2) \sin \Big (t \log \Big (P^-(d_1)-\frac{1}{2} \Big ) \Big )  \frac{dt}{t}\\
&  \quad + O(Y) \\
&  = 2\delta M(\cB)+ O(Y)
\end{align*}
where the last line follows from the truncated Perron formula once again. The result follows immediately.
\end{proof}

The next sieve estimate is based on an idea which dates back to Vinogradov who applied it to estimate sum over primes. The idea is to systematically take out the largest prime factor until we get sums which we can estimate by our Type II estimate (Lemma~\ref{type two estimate}).

\begin{lemma}   \label{1/4}
For $2 \le y <z < x^{1/4 }$, we have
\begin{equation*} 
\sum_{ y \le p \le z} \cS (\cA_{p}; y,p) =2\delta \sum_{ y \le p \le z} \cS (\cB_{p}; y,p) + O(Y).
\end{equation*}
\end{lemma}

\begin{proof}
By the method of Lemma~\ref{lem:Buchstab} we have
\begin{align*}
\sum_{y \le p \le z} \cS(\cA_p; y,p) & = \sum_{y \le p_1 \le p \le z} \cS(\cA_{pq_1}; y, p_1) +O( x^{1/2} ). 
\end{align*}
We will first consider the sum on the right with $p p_1 >x^{1/4}$. Clearly we have
$$
\sum_{\substack{y \le p_1 \le p \le z \\ pp_1 >x^{1/4} }} \cS(\cA_{pp_1}; y, p_1) = \sum_{\substack{y \le p_1 \le p \le z \\ pp_1 >x^{1/4 }}} \sum_{\substack{npp_1 \in \cA \\ (n, P(y)P(p_1,x] )=1 }} 1.
$$
Set $m = pp_1$ and note that $m < x^{1/2}$. It follows that
$$
\sum_{\substack{y \le p_1 \le p \le z \\ pp_1 >x^{1/4} }} \cS(\cA_{pp_1}; y, p_1) = \sum_{\substack{mn \in \cA \\ x^{1/4} < m < x^{1/2} }} a_m \sum_{\substack{d_1d_2|n \\ d_1|P(y) \\ P^-(d_2) > P^-(m)}} \mu(d_1 d_2)
$$
where $a_m$ is 1 if $m=pp_1$ with $p,p_1 \in [y,z]$ and 0 otherwise. Applying the method of Lemma~\ref{1/4 3/4}, we get
$$
\sum_{\substack{y \le p_1 \le p \le z \\ pp_1 >x^{1/4} }} \cS(\cA_{pp_1}; y, p_1) = 2\delta \sum_{\substack{y \le p_1 \le p \le z \\ pp_1 >x^{1/4} }} \cS(\cB_{pp_1}; y, p_1) +O( Y ).
$$
The only part left to consider is the sum
$$
\sum_{ \substack{ y \le p_1 \le p \le z \\ pp_1 \le x^{1/4} }} \cS(\cA_{pp_1}; y, p_1).
$$
If it is zero then we are done, otherwise we take out the next largest prime factor to get
$$
\sum_{ \substack{ y \le p_1 \le p \le z \\ pp_1 \le x^{1/4} }} \cS(\cA_{pp_1}; y, p_1) = \sum_{\substack{ y \le p_2 \le p_1 \le p \le z \\ pp_1 \le x^{1/4} }}  \cS(\cA_{pp_1 p_2}; y , p_2) +O(x^{1/2 }).
$$ 
The sum on the right with $pp_1 p_2 >x^{1/4 }$ can be dealt with again by the method of Lemma~\ref{1/4 3/4}. By induction this can go on for at most $O(\log x)$ steps. Since we have an asymptotic formula for every sum, the result follows.
\end{proof}

For our next sieve estimate, we note that our Type II estimates are not sufficient in this region. We bypass this complication by a role reversal that minimises the length of summation in exchange for sifting primes.

\begin{lemma}   \label{3/4}
For $x^{{3/4}} \le z \le x$, we have
\begin{equation*} 
\sum_{ x^{3/4} < p \le z} \cS (\cA_{p}; y,p) = 2\delta \sum_{ x^{3/4} < p \le z} \cS (\cB_{p}; y,p)   +O(Y).
\end{equation*}
\end{lemma}

\begin{proof}
Take $x^{3/4} < p \le z$ and an element $np \in \cA$ such that if $q|n$ where $q$ is prime then $q \in [y,p]$. This gives
\begin{equation} \label{switching roles}
\sum_{ x^{3/4} < p \le z} \cS (\cA_{p}; y,p) = \sum_{n < x^{1/4} } c_{n} \cS(\cA_n' ; (x/n)^{1/2 }) + O(x^{o(1)}).
\end{equation}
Here
\begin{align*}
\cA'_n & = \{ \max \{ y+\Delta(y),x^{3/4} \} < m \le z : mn \in \cA \} \\
& = \{ \max \{ y+\Delta(y),x^{3/4} \} < m \le \min \{ z, x/n\} : \lVert \alpha mn + \beta \rVert < \delta \}
\end{align*}
where $c_n$ is 1 if all prime factors of $n$ are at least $y$ and 0 otherwise and $\Delta(y)$ is -1/2 if $y$ is an integer and zero otherwise. We also recall
$$
\cS(\cW; k)= \# \{ w \in \cW : (w,P(k))=1 \}.
$$

We note that the sum in~(\ref{switching roles}) may be empty depending on the choice of $y$. To show~(\ref{switching roles}), take $m \in S(\cA_n' ; (x/n)^{\frac{1}{2}})$ and suppose that $m$ has two prime factors say $q_1, q_2$ then we must have
$$
m \ge q_1 q_2 \ge (x/n)^{1/2}(x/n)^{1/2} = x/n.
$$
This is a contradiction unless $mn =x$ and this occurs  at most $O(x^{o(1)})$ times. Moreover, it is evident that
\begin{align*}
\sum_{n < x^{1/4} } c_n \cS(\cA_n' ; (x/n)^{1/2 }) &= \sum_{n < x^{1/4}} c_n (\cS(\cA'_n ; x^{1/2 }) +O(x^{1/2 })) \\
& = \sum_{n < x^{1/4}} c_n \cS(\cA'_n ; x^{1/2}) +O(x^{3/4 }).
\end{align*}

Let $n \sim N$ where $N \ll x^{1/4}$ and $\max \{ y + \Delta(y),x^{3/4}\} n \ll M(n)=M \ll \min \{zn,x\}$ and consider the set
$$
\cA^{(M)}= \{ m \sim M : \lVert \alpha m +\beta \rVert < \delta  \}.
$$
Then we have
$$
\cA_n^{(M)} = \{ m \sim M/n : \lVert \alpha mn + \beta \rVert < \delta \}
$$
and by the Harman sieve~\cite[Lemma 2]{H2} using the  Type I and II estimate (Lemma~\ref{type one estimate} and~\ref{type two estimate}), we get
$$
\sum_{n \sim N} c_n  \cS(\cA_n^{(M)}; x^{1/2}) = 2 \delta  \sum_{n \sim N} c_n \cS(\cB_n^{(M)}; x^{1/2 }) + O(Y).
$$
Therefore by summing over $N,M$ we obtain
$$
\sum_{n < x^{1/4}} c_n \cS(\cA'_n; x^{1/2}) = 2\delta  \sum_{n < x^{1/4}} c_n  \cS(\cB'_n; x^{1/2}) + O(Y),
$$
which is what we needed to show.
\end{proof}

For a set $\cC \subseteq [2,x]$ of integers and for integers $2 \le y < z \le x$, we denote
$$
\cT(\cC;y,z) = \# \{ c \in \cC : \mu^2(c)=1, p|c \implies p\in [y,z] \}.
$$

The next three results are square-free analogue of Lemma~\ref{1/4 3/4}-\ref{3/4}.

\begin{lemma} \label{1/4 3/4 sf}
For $x^{ 1/4 } \le y < z \le x^{3/4 }$, we have
\begin{equation*} 
\sum_{y  \le p \le z} \cT (\cA_{p}; y,p-1)  = 2\delta \sum_{y  < p \le z} \cT (\cB_{p}; y,p-1) +  O(Y).
\end{equation*}
\end{lemma}

\begin{proof}
By writing
\begin{align*}
\sum_{y \le p \le z} \cT(\cA_p;y,p-1)   &=\sum_{y \le p \le z} \sum_{ \substack{np \in \cA \\ (n,P(y),P(p-1,x])=1} } \mu^2(np) \\
& = \sum_{\substack{mn \in \cA \\ y \le m \le z}} \mathbbm{1}_{\mathbb{P}}(m) \mu^2(n) \sum_{\substack{d_1d_2|n \\d_1|P(y) \\ P^-(d_2) \ge m}} \mu(d_1 d_2).
\end{align*}
 We see that the method of Lemma~\ref{1/4 3/4} gives the result.
\end{proof}

\begin{lemma} \label{1/4 sf}
For $2 \le y <z < x^{1/4 }$, we have
$$
\sum_{y \le p \le z} \cT(\cA_p;y,p-1) = 2\delta \sum_{y \le p \le z} \cT(\cB_{p};y,p-1) +O(Y).
$$
\end{lemma}

\begin{proof}
We follow as in Lemma~\ref{1/4} but we successively take out the largest distinct prime factors.
\end{proof}

\begin{lemma}   \label{3/4 sf}
For $x^{3/4} \le z \le x$, we have
\begin{equation*} 
\sum_{ x^{3/4} < p \le z} \cT (\cA_{p}; y,p-1) = 2\delta \sum_{ x^{3/4} < p \le z} \cT (\cB_{p}; y,p-1)   +O(Y).
\end{equation*}
\end{lemma}

\begin{proof}
This follows immediately from the proof of Lemma~\ref{3/4} by taking $c_n$ there to be 1 if $n$ is square-free and all prime factors of $n$ are at least $y$ and 0 otherwise.
\end{proof}

\section{Proof of Theorem~\ref{smooth}}
 
By Dirichlet's theorem, we have $|\alpha - a/q|<1/q^2$ for infinitely many $q$, let $(q_k)_{k \in \mathbb{N}}$ be an increasing sequence of such denominators. Then we can choose an increasing sequence $(x_k)_{k \in \mathbb{N}}$ so that $q_k=x_k^{2/3 }$. Now we take $x_k$ to be a sufficiently large element in the sequence $(x_k)_{k \in \mathbb{N}}$. 

By Lemma~\ref{lem:Buchstab}, we assert
$$
\cS(\cA;y,z) = \sum_{y \le p \le z} \cS(\cA_p;y,p) + \cS(\cC; x_k^{ 1/2 }) +O(x_k^{ 1/2 }).
$$
By~\cite[Theorem 2]{V} and partial summation, we have
$$
\cS(\cC; x_k^{1/2}) = 2\delta \cS(\cD;  x_k^{ 1/2 }) +O(Y)
$$
where $\cD= \{ n \in \mathbb{N} : y\le n \le z\}$ and $Y =x_k^{3/4 + \varepsilon/2 + o(1) }$. Therefore by substituting this into the above, we obtain
$$
\cS(\cA;y,z) = \sum_{y \le p \le z} \cS(\cA_p;y,p) + 2\delta \cS(\cD; x_k^{ 1/2 }) +O(Y).
$$
Suppose $y < x_k^{1/4 }$. Then we write
\begin{align*}
\cS(\cA;y,z) &= \Bigg (\sum_{y \le p \le z <x_k^{1/4}} + \sum_{y \le p < x_k^{1/4}} + \sum_{x_k^{1/4}\le p \le z \le x_k^{3/4}} +  \sum_{x_k^{1/4}\le p \le x_k^{3/4} \le z } \\
&\quad + \sum_{x_k^{3/4} < p \le z} \Bigg ) \cS(\cA_p;y,p)  + 2\delta \cS(\cD; x_k^{\frac{1}{2} }) +O(Y).
\end{align*}
We note that some of the sums above may be empty depending on the choice of $z$. Applying Lemma~\ref{1/4 3/4}-\ref{3/4} then Lemma~\ref{lem:Buchstab} to the above we get 
$$
\cS(\cA;y,z) = 2\delta \cS(\cB ;y,z) +O(Y).
$$
The case $x_k^{1/4 } \le y < x_k^{1/2 }$ is similar to the case above and the first result follows. 

Now suppose $\alpha \in \cI(A)$ then we can assume $\alpha \in (0,1) \backslash \mathbb{Q}$. Indeed, we can write $\alpha = b + r$ where $b \in \mathbb{Z}$ and $r \in (0,1) \backslash \mathbb{Q}$ so that $\lVert \alpha n + \beta \rVert = \lVert r n + \beta \rVert$. Therefore we have the continued fraction expansion 
$$
\alpha = [0; a_{1}, a_2, \ldots]
$$ 
with $0 < a_{i} \le A$ for some absolute constant $A$. The convergents $p_{s}/q_{s}$ to $\alpha$ satisfy the inequality
$|\alpha -p_{s}/q_{s}| < 1/q_s^2$.
Clearly $( q_{s} )_{s \ge 1}$ is a strictly increasing sequence. Therefore, for any sufficiently large $x$ there exist $s \in \mathbb{N}$ such that $q_{s-1} \le x^{2/3} \le q_{s}$. Since $q_{s}=a_{s}q_{s-1}+q_{s-2}$ for $s \ge 2$, we get $q_{s} \le  (A+1)q_{s-1} \le  2A x^{2/3} \ll x^{2/3}$, because $A$ is fixed. Hence  $x^{2/3} \ll q_{s} \ll x^{2/3}$ and the result follows by the argument above.

\section{Proof of Theorem~\ref{y-z smooth}}

We proceed as in the proof of Theorem~\ref{smooth} but instead we apply Lemma~\ref{1/4 3/4 sf}-\ref{3/4 sf}.

\section*{Acknowledgement}
The author is grateful to the anonymous referee for their excellent comments which resulted in a complete rewrite of this paper, and to B. Kerr and I. E. Shparlinski for helpful comments on a previous draft. Lastly the author thanks T. Evans and N. Poznanovi\'c for stimulating conversations on a related topic at the beach.

 The author was supported by an Australian Government Research Training Program (RTP) Scholarship.


\end{document}